\documentclass{article} \usepackage{amsthm,amsmath,amssymb,mathrsfs,url,amsfonts} 
\usepackage{amsmath, verbatim, amsfonts, amssymb} \usepackage{mathrsfs}

\usepackage{amsfonts}\usepackage{amssymb} \usepackage{stackrel} 
\usepackage{bm}
\usepackage{latexsym}\usepackage{epsfig}\usepackage{graphicx}\usepackage{oldgerm}

\newcommand\DN{\newcommand} 
 
\numberwithin{equation}{section}
\newcounter{Const} 
\DN\Ct{\refstepcounter{Const}c_{\theConst}}
\DN\cref[1]{c_{\ref{#1}}}	
 \theoremstyle{definition}

 \newtheorem{theorem}{Theorem}[section]
 \newtheorem{lemma}[theorem]{Lemma}

\DN\lref[1]{Lemma~\ref{#1}}\DN\tref[1]{Theorem~\ref{#1}}\DN\pref[1]{Proposition~\ref{#1}}
\DN\sref[1]{Section~\ref{#1}}
\DN\ssref[1]{Subsection~\ref{#1}}
\DN\dref[1]{Definition~\ref{#1}} 
\DN\rref[1]{Remark~\ref{#1}} 
\DN\corref[1]{Corollary~\ref{#1}}
\DN\eref[1]{Example~\ref{#1}}

\DN\bs{\bigskip}\DN\ms{\medskip}
\DN\N{\mathbb{N}}\DN\R{\mathbb{R}}\DN\Q{\mathbb{Q}}\DN\C{\mathbb{C}}
\DN\Z{\mathbb{Z}}

\DN\map[3]{#1\!:\!#2\!\to\!#3}
\DN\ot{\otimes} 
\DN\PD[2]{\frac{\partial#1}{\partial#2}}
\DN\half{\frac{1}{2}}
\DN\elaw{\stackrel{\mathrm{law}}{=}}
\DN\eac{\stackrel{\mathrm{ac}}{\sim}}

\DN\Rd{\R ^d}
\DN\RtwoN{\R ^{2\N }}
\DN\Rtwo{\R ^2}
\DN\Rtwom{\R ^{2m}}
\DN\RZ{\R ^{\Z }}

\DN\RdN{\R ^{d\N }} 

\DN\limi[1]{\lim_{#1\to\infty}} \DN\limz[1]{\lim_{#1\to0}}
\DN\limsupi[1]{\limsup_{#1\to\infty}}\DN\liminfi[1]{\liminf_{#1\to\infty}} 	
\DN\limsupz[1]{\limsup_{#1\to 0}}\DN\liminfz[1]{\liminf_{#1\to 0}}

\DN\LRm{L^2 (\SSSR , \Deltam d\x )}
\DN\ERxim{\bm{\E } _{\Rxi }^{\m }}
\DN\ERxi{\bm{\E } _{\Rxi }} 

\DN\Jt{J _t}
\DN\Lmlocone{L_{\mathrm{loc}}^{1} (\R \times \sSS , \muone )}
\DN\Lmloctwo{L_{\mathrm{loc}}^{2} (\R \times \sSS , \muone )}

\DN\mmm{\mm \Rxim }
\DN\Lmmm{L^2 (\SSSR, \mmm d\x )}

\DN\domWR{\bm{{\dom }}_{\rR }^m }
\DN\pWR{{p}_{\rR }^m }
\DN\EWRD{{\bm{\E }}_{\rR }^m }
\DN\TW{\bm{T} _{\rR }^{m}}

\DN\domlRxim{\bm{\dom }\Rxim }
	\DN\dlRxim{\doml \Rxim }
\DN\domlRxi{\doml _{\Rxi }}

\DN\domRxim{\dom \Rxim }

\DN\Rxi{\rR , \xi }\DN\Rxim{_{\Rxi }^{\m }}
\DN\Rximxi{_{\Rxi }^{\mxi }}
\DN\Rpixi{\rR , \piRc (\xi )}

\DN\fRxi{f_{\Rxi }}
\DN\gRxi{g_{\Rxi }}
\DN\TTR{ \mathscr{T}_{\rR }}
\DN\TT{\mathscr{T} }
\DN\Tl{\underline{T}}
\DN\Tlt{\Tl (t)}
\DN\Tlu{\Tl (u)}
\DN\TlR{\Tl _{\rR }}
\DN\TlRt{\TlR (t)}
\DN\TlRu{\TlR (u)}
\DN\Tt{T (t)} \DN\Tu{T (u)}

\DN\TRxim{\bm{T } \Rxim }

\DN\Deltam{\Delta ^m}

\DN\Ro{\rR }
\DN\SSSR{\mathbf{S}_{\Ro }^{\m }}
\DN\SSSRone{\mathbf{S}_{\Ro }^{1 }}
\DN\mm{\mathbf{m}}
\DN\limR{\limi{\rR }}
\DN\DDDR{\mathbb{D}_{\rR }}
\DN\DDDRO{\mathbf{D}_{\Ro }}
\DN\f{\check{f}} \DN\g{\check{g}} 
\DN\si{s^i} \DN\sj{s^j}
\DN\sii{s^{i+1}} \DN\sjj{s^{j+1}}
\DN\ER{\E _{\rR }}
 \DN\ERl{\underline{\E }_{\rR }}

\DN\As[1]{$ ($\textbf{#1}$)$} 

\DN\AAAA{\mathscr{A}}
\DN\BBBB{\mathscr{B} }\DN\BB{\mathfrak{B}}
\DN\AAA{\mathfrak{A}}

\DN\Am{\mathbf{A}_0^{\m }} \DN\Bm{\mathbf{B}_t^{\m }}
\DN\Ams{\mathbf{A}_0^{\m * }} \DN\Bms{\mathbf{B}_t^{\m * }}

\DN\m{m}\DN\nnn{n}
\DN\mxi{\m (\RxiC )} 
\DN\xiR{\xi , \rR }
\DN\RxiC{\piRc (\xi )}
\DN\1{\El _{\Rxi }^{m }}

\DN\3{\El _{\rR } }

\DN\5{1_{\Ams }(\wms _0)} 
\DN\Bmsws{1_{\Bms }(\wms _t)} 
\DN\6{1_{\Ams }(\wms _0 ) 1_{\Bms }(\wms _t )}
\DN\ABm{\w _0 \in \Am , \w _t \in \Bm }
\DN\wA{\w _0 \in \mathbf{A}_0} \DN\wB{ \w _t \in \mathbf{B}_t}
\DN\7{\wA ,\wB }

	\DN\pPi{ \pP ^{\infty}}
	\DN\pPix{\pPi _{\x }}

	\DN\sumjm{\sum_{ \jmm }}

\DN\X{\mathbf{X}}
\DN\Xm{\X ^{ \m }}

\DN\YmX{(\mathbf{Y}^{\m } , \X ^{\m *})}
\DN\XmX{(\X ^{\m } , \X ^{\m *})}
\DN\Xms{ \X ^{\m *} }

\DN\XM{\X ^{[\m ]}}

\DN\XX{\mathfrak{X}}
\DN\xx{\mathfrak{x}}
\DN\xxt{\xx _t}
\DN\XXms{\XX ^{\m *}}
\DN\wz{\w _0}
\DN\w{\mathbf{w}} \DN\wm{\w ^{ \m }} \DN\wms{\w ^{ \m * }} \DN\wmm{\w ^{[\m ]}}
\DN\x{\mathbf{x}} 
\DN\xm{\x } 	

\DN\xms{\x ^{ \m *}}
\DN\y{\mathbf{y}} \DN\ym{\y ^{ \m }}

\DN\sss{\mathfrak{s}}
\DN\K{\mathrm{K}}

\DN\bwm{b _{\w }^{ \m }} \DN\bwmi{b _{\w }^{ \m , i }}

\DN\OTwm{ \mathscr{O}_{T,\w }^{ \m }}
\DN\OTwme{ \mathscr{O}_{T,\w }^{ \m , \epsilon }}
\DN\OTwmz{ \mathscr{O}_{T,\w }^{ \m , 0 }}
\DN\OTwmee{ \mathscr{O}_{T,\w }^{ \m , \epsilon /2}}
\DN\QTwme{ \mathscr{Q}_{T,\w }^{ \m , \epsilon }}

\DN\RZZ{\R _< ^{\mathbb{Z}}} 
\DN\Rm{\R _< ^{\m }} \DN\Rms{\R _<^{ \m *}} \DN\Rmm{\R _<^{[\m ]}}
\DN\W{\mathrm{W}}
\DN\WW{C(\R ; \sSS )}
\DN\ww{\mathfrak{w}}
\DN\wwz{\mathfrak{w}_0}
\DN\wwt{\mathfrak{w}_t}

\DN\Pm{\pP ^{\m }}
\DN\Pmx{\Pm _{\x }}
\DN\Pwm{\pP _{\w }^{\m }}

\DN\PP{\mathsf{P}}
 \DN\PPm{\PP _{\mu }}
\DN\PPx{\PP _{\xx }}
\DN\PPwm{\PP _{\w }^{\m }}

\DN\mul{\mu \circ {\lab }^{-1}}
\DN\mui{\mu ^{\infty}} 

\DN\muR{\mu \circ \piR }
\DN\muRxim{\mu _{\Rxi }^{\m }}
\DN\muRximxi{\mu _{\Rxi }^{\mxi }}
\DN\muRxi{\mu _{\Rxi }}

\DN\piRc{\pi _{\rR }^c}
\DN\piR{\pi _{\rR }}

	\DN\nN{N}
	\DN\pP{P}
	\DN\qQ{Q}

\DN\rR{R} 
\DN\sS{S} \DN\sSS{\mathfrak{S}}

\DN\pPm{\pP _{\mu }}
\DN\pPs{\Pts }
 \DN\KK{\mathcal{K} _{\mathrm{sin}, 2}}

 \DN\muone{\mu ^{[1]}}

\DN\Sm{\sS ^{ \m }} \DN\SmSS{\Sm \ts \sSS }

\DN\SSrm{\sSS _r^{ \m }}
\DN\SSRm{\SSR ^{ \m }}
\DN\SSR{\sSS _{\rR }}

\DN\Ssi{\sSS _{\mathrm{s,i}}}

\DN\SR{\sS _{\rR }}
\DN\SRm{\sS _{\rR }^{ \m }}

\DN\dlog{\mathfrak{d}}
\DN\dmu{\dlog ^{\mu }}

\DN\dR{\dom _{\rR }}
\DN\dRl{\underline{\dom }_{\rR }}
\DN\dom{\mathcal{D}} 
\DN\doml{\underline{\dom }} 
\DN\di{\dom _{\circ }}
\DN\diR{\dom _{\rR ,\circ }} 
\DN\Lm{L^2(\sSS ,\mu )}
\DN\Emu{\mathcal{E}^{\mu } }
\DN\E{\mathcal{E} }
\DN\El{\underline{\E }}

\DN\ulab{\mathfrak{u} }
\DN\lab{\mathfrak{l} } 

\DN\labm{\lab _{\m } } 
\DN\labi{\lab ^i}
\DN\labR{\lab _{\rR }^{\m }}

\DN\ulabm{\x ^{[\m ]}}

\DN\upath{\ulab _{\mathrm{path}}}
\DN\lpath{\lab _{\mathrm{path}}}

\begin{document}
\title{Ergodicity of unlabeled dynamics of Dyson's model in infinite dimensions} 

\author{Hirofumi Osada and Shota Osada}

\maketitle

\pagestyle{myheadings}
\markboth{Hirofumi Osada; Shota Osada}
{Ergodicity of unlabeled dynamics of Dyson's model in infinite dimensions}
\maketitle

\begin{abstract} 
Dyson's model in infinite dimensions is a system of Brownian particles 
that interact via a logarithmic potential with an inverse temperature of $ \beta = 2$. 
The stochastic process can be represented by the solution to an infinite-dimensional stochastic differential equation. 
The associated unlabeled dynamics (diffusion process) are given by the Dirichlet form with the sine$ _2$ point process as a reference measure. 
In a previous study, we proved that Dyson's model in infinite dimensions is irreducible, but left 
the ergodicity of the unlabeled dynamics as an open problem. 
In this paper, we prove that the unlabeled dynamics of Dyson's model in infinite dimensions are ergodic. 
\bs 


\ms 
\end{abstract}

\section{Introduction} \label{s:1}

Dyson's Brownian motion is given by the solution to the following finite-dimensional stochastic differential equation (SDE): 
\begin{align}\label{:10a}&
 X_t^{\nN ,i} - X_0^{\nN ,i} = B_t^i + 
 \frac{\beta }{2}
 \int_0^t
\sum_{ j\not= i}^{\nN }
 \frac{1}{X_u^{\nN ,i} - X_u^{\nN ,j}} du - 
 \frac{\beta }{2\nN }
 \int_0^t
 \frac{1}{X_u^{\nN ,i}} du 
\end{align}
 for $ \beta = 1,2,4 $. 
If $ \beta = 2$, then SDE \eqref{:10a} describes the dynamics of the eigenvalues of Gaussian unitary ensembles of order $ \nN \in \N $ \cite{Dys62,Meh04}. 
By taking $ \beta = 2 $ and letting $ \nN \to \infty $ in Eq.\,\eqref{:10a}, 
Spohn introduced the infinite-dimensional stochastic differential equation (ISDE)
\begin{align}& \label{:10b} 
 X_t^i - X_0^i = B_t^i + \int_0^t \sum_{ j\not= i , \, j \in \Z } 
 \frac{1}{X_u^i - X_u^j } du \quad (i\in\mathbb{Z})
.\end{align}
Spohn called ISDE \eqref{:10b} \lq\lq Dyson's model'' (of interacting Brownian particles) 
\cite{Spo87,Spo.tracer}. 
To emphasize that there are an infinite number of dimensions, we call this ISDE \lq\lq Dyson's model in infinite dimensions.''

Dyson's model in infinite dimensions with an inverse temperature of $ \beta \ge 0 $ 
is an $ \RZ $-valued stochastic process of Brownian particles interacting via a logarithmic potential. 
The stochastic process is given by the ISDE 
\begin{align}& \label{:10e} 
 X_t^i - X_0^i = B_t^i + 
 \frac{\beta }{2} \int_0^t
 \lim_{\rR \to\infty }\sum_{|X_u^i - X_u^j |< \rR , \ j\not= i}^{\infty} 
 \frac{1}{X_u^i - X_u^j } du \quad (i\in\mathbb{Z})
.\end{align}
Because the number of particles is infinite, the meaning of the sum in the drift term is ambiguous. 
The long-range nature of the logarithmic interaction indicates that the sum represents the conditional convergence of the one-reduced Campbell measure. 
We formulate ISDE \eqref{:20a} in a strict sense using the concept of the logarithmic derivative 
$ \dmu $. (See \cite{o.isde} for further details.)

Spohn \cite{Spo87} constructed the limit dynamics 
as the $ L^2 $-Markovian semi-group given by the Dirichlet form on $ \Lm $, defined as 
\begin{align}\label{:10c}&
\mathcal{E} ( f , g ) = \int _{\sSS } \mathbb{D} [ f , g ] d\mu 
,\end{align}
where $ \sSS $ is the configuration space over $ \R $, $ \mu $ is the sine$ _{2}$ random point field, and $ \mathbb{D} $ is the standard carr\'{e} du champ on $ \sSS $ such that 
\begin{align}& \notag 
\mathbb{D} [f , g ] (\sss )= \half \sum_i \PD{\check{f}}{s^i} \PD{\check{g}}{s^i} 
.\end{align}
Here, for a function $ f (\sss ) $ on $ \sSS $, 
$ \check{f} (\mathbf{s}) $ is the symmetric function on 
$ \sum_{m=0}^{\infty} \R ^m $ such that 
$ \check{f} (\mathbf{s}) = f (\sss )$, $ \sss = \sum_i \delta_{s^i}$. 
Furthermore, the domain of the Dirichlet form is taken to be the closure of the polynomials on 
$ \sSS $. 
The sine$ _{2}$ random point field $ \mu $ is 
a determinantal random point field on $ \R $ 
whose $ \m $-point correlation function $ \rho ^{ \m }$ 
with respect to the Lebesgue measure is given by 
\begin{align} &\notag 
\rho ^{ \m }(\mathbf{x})= \det [ \KK ( x^i , x^j ) ] _{ i , j = 1 }^{ \m }
.\end{align}
Here, for a constant $ 0 < \rho < \infty $, we set the sine kernel $ \KK $ such that 
\begin{align} \label{:10f}&
 \KK ( x , y ) = \frac{\sin \{\rho (x-y)\} }{\pi (x-y)}
.\end{align}
Spohn \cite{Spo87} proved the closability of $ \mathcal{E} $ on $ \Lm $ with a predomain consisting of polynomials on $ \sSS $. 

In \cite{o.dfa,o.rm}, the first author proved that 
$ (\mathcal{E}, \di ^{\mu })$ is closable on $ \Lm $, 
and that its closure $ (\mathcal{E}, \dom )$ is a quasi-regular Dirichlet form. 
Here, $ \di $ is the set consisting of local and smooth functions on $ \sSS $. 
We take $ \di ^{\mu }$ such that 
\begin{align} & \notag 
\di ^{\mu } = \{ f \in \di \, ;\, \mathcal{E} ( f , f ) < \infty ,\, f \in \Lm \} 
.\end{align}
Thus, the $ L^2$-Markovian semi-group was constructed alongside the diffusion 
\begin{align} & \label{:10g}
 \XX _t = \sum_{i \in \mathbb{Z}} \delta _{X^i _t }
\end{align}
associated with the Dirichlet form $ (\mathcal{E}, \dom )$ on $ \Lm $. 
We call $ \XX $ the unlabeled dynamics or unlabeled diffusion 
because the state space of the process is $ \sSS $. 

Let $ \mu _{\beta }$ be the sine$ _{\beta }$ random point field \cite{v-v}. 
Replacing $ \mu $ by $ \mu _{\beta }$ ($ \beta = 1,4$), 
we consider the Dirichlet form $ \mathcal{E} $ in Eq.\,\eqref{:10c} for $ \beta =1,4$. 
The unlabeled diffusion has been constructed for $ \beta = 1,4$ \cite{o.rm}, and 
the associated labeled process $ \X = (X^i)_{i\in \N }$ 
satisfies ISDE \eqref{:10e} for $ \beta = 1,2,4$ \cite{o.isde}. 
These cases have been proved as examples of the general theory developed in various papers \cite{o.tp,o.isde,o.rm,o.rm2}. 

In \cite{o.isde}, the meaning of a solution to an ISDE is a weak solution; 
the uniqueness of such a weak solution and its Dirichlet form were left open in \cite{o.isde,o.rm}. 
A weak solution $ (\X ,\mathbf{B})$ can be loosely described as a pair consisting 
 of the stochastic process $ \X $ and the Brownian motion $ \mathbf{B} $ satisfying 
the ISDE. 
A strong solution is a weak solution $ (\X ,\mathbf{B})$ 
such that $ \X $ is a function of the Brownian motion $ \mathbf{B} $ 
and the initial starting point $ \x $. 
(See \cite{IW,o-t.tail} for the concept of strong and weak solutions of SDEs.) 

Tsai \cite{tsai.14} solved ISDE \eqref{:10e} for all $ \beta \in [1,\infty )$. 
He proved the existence of a strong solution and the path-wise uniqueness of this solution. 
The method used by Tsai depends on an artistic coupling specific to Dyson's model. 
A non-equilibrium solution is obtained in the sense that the ISDE is solved by starting at each point 
in an explicitly given subset $ \sSS _0 \subset \sSS $ such that $ \mu (\sSS _0 ) = 1 $. 

The $ \mu $-reversibility of the associated unlabeled diffusion was left open in \cite{tsai.14}. 
Combining \cite{o.isde} and \cite{tsai.14}, we find that the unlabeled process given by the solution of Eq.\,\eqref{:10b} obtained in \cite{tsai.14} is reversible with respect to $ \mu _\beta $ for $ \beta = 1, 4$. 
For a general $ \beta > 0 $, we expect that the reversible probability measure of the unlabeled diffusion given by the solution to ISDE \eqref{:10e} is the sine$_{\beta } $ random point field. 
This remains an open problem, except for $ \beta = 1,2,4 $ \cite{o.rm}.

The first author and Tanemura \cite{o-t.tail} also proved the existence of a strong solution and the path-wise uniqueness of this solution for $ \beta = 1,2,4$. 

Using the result in \cite{o-t.tail}, Kawamoto {\em et al.} proved the uniqueness of Dirichlet forms \cite{k-o-t.udf}. 
In proving the uniqueness of Dirichlet forms, they examined the condition of an infinite system of finite-dimensional SDEs with consistency \cite{k-o-t.ifc}, 
which plays an important role in the theory developed in \cite{o-t.tail}. 
Kawamoto and the first author derived a solution to the ISDE based on $ \nN $-particle systems \cite{k-o.fpa,k-o.sg,k-o.du}. 

\bs 

In the remainder of this paper, we consider the case $ \beta = 2$. 
Hence, we take $ \mu $ to be the sine$ _2$ random point field. 

In \cite{o-tu.irr}, the first author and Tsuboi proved that the labeled process $ \X $ is irreducible (see \lref{l:21}). 
Thus, it is natural to consider the existence of the invariant probability measure $ \nu $ for the labeled process $ \X $ and the ergodicity of the stationary labeled process $ \X _0 \elaw \nu $. 
These two problems were left open in \cite{o-tu.irr}.

Let $ \ulab $ be the unlabeling map such that 
$ \ulab ( \mathbf{s} ) = \sum_{i\in\mathbb{Z}} \delta_{\si }$ for $ \mathbf{s}=(s^i)$. 

In \cite[Proposition]{Spo.tracer}, Spohn proved that tagged particles exhibit logarithmic asymptotic behavior as $ t \to \infty $ 
such that, with the constant one-point correlation function $ \rho $ in Eq.\,\eqref{:10f}, 
\begin{align}\label{:10m}&
\limi{t} \frac{1}{ \log t } E [|X_t^i-X_0^i|^2] = (\pi \rho )^{-2}
.\end{align}
This result suggests that there is no invariant probability measure $ \nu $ of $ \X $ satisfying $ \mu = \nu \circ \ulab ^{-1}$, 
which implies that $ \X $ is not ergodic in the sense that $ \X $ has no invariant probability measure. 
Hence, we consider the ergodicity of the unlabeled diffusion $ \XX $ in Eq.\,\eqref{:10g} 
associated with $\X $. 

Let $ \X = (X^i)_{i\in\Z }$ be a solution of \eqref{:10e}. 
Let $ \XX $ be such that $ \XX _t = \sum_{i\in\Z }\delta_{X_i^i}$. 
The goal of this paper is to prove that the $ \mu $-reversible diffusion $ (\XX ,\PPm ) $ 
associated with $ \X $ is ergodic under a time shift (\tref{l:14}). 
To prove this, we show that 
an $ \E $-harmonic function is constant (\tref{l:12}), 
and that $ \mu $ is extremal in the space of invariant probability measures of $ \XX $ (\lref{l:51}). 

\bs 

By definition, the configuration space $ \sSS $ over $ \mathbb{R}$ is given by 
\begin{align} &\notag 
\sSS = \Big\{ \sss = \sum _i \delta_{\si }\, ;\, \sss ( K ) < \infty \ \text{ for any compact } K \Big\} 
.\end{align}
We endow $ \sSS $ with the vague topology. 
Under the vague topology, $ \sSS $ is a Polish space. 
A probability measure on $ (\sSS , \mathscr{B}(\sSS ) )$ is called a random point field. Let 
\begin{align} &\notag 
\Ssi = \Big\{ \sss \in \sSS \, ;\, \sss (\{ s \} ) \le 1 \text{ for all } s \in \R ,\, \sss (\R ) = \infty \Big\} 
.\end{align}
In \cite{o.col,o.rm}, we proved that the sine$_2 $ random point field $ \mu $ satisfies 
\begin{align}\label{:11m}&
\mathrm{Cap} ((\Ssi )^c) = 0 
,\end{align}
where $ \mathrm{Cap}$ denotes the capacity given by the Dirichlet form 
$ (\mathcal{E}, \dom )$ on $ \Lm $. 
The result has a dynamic interpretation. 
Indeed, using Dirichlet form theory, we can deduce from Eq.\,\eqref{:11m} that, for $ \mu $-a.s.\,$ \xx $, 
\begin{align}\label{:11n}&
\PPx (\ww _t \in \Ssi \text{ for all }t ) = 1 
.\end{align}
Here $ \{ \PPx \} $ is the diffusion associated with the Dirichlet form 
$ (\mathcal{E}, \dom )$ on $ \Lm $, and $ \ww =\{ \wwt \} \in C([0,\infty), \sSS )$. 

We write $ \mathbf{s}=({\si })_{i\in\mathbb{Z}} \in \R ^{\mathbb{Z}}$, and set 
\begin{align} &\notag 
\RZZ = \{ \mathbf{s}=({\si })_{i\in\mathbb{Z}}\in \R ^{\mathbb{Z}} \, ;\,{\si }< {\sii } 
\ \text{ for all }\ i \, \} 
.\end{align}
Let $ \map{\lab }{\Ssi }{\RZZ } $ be a function such that $ \ulab \circ \lab = \mathrm{id.}$. 
We call $ \lab $ a labeling map. Note that many labeling maps exist; in the remainder of this paper, we fix a labeling map $ \lab $.

Let $ \X = ( X^i)_{i\in \mathbb{Z} }$ be a solution to ISDE \eqref{:10e} 
with $ \beta = 2 $ defined on a filtered space $ (\Omega , \mathscr{F}, P , \{ \mathscr{F}_t \} )$ 
such that $ \X _0 \elaw \mul $. 
Then it is known \cite{o-t.tail} that
\begin{align*}&
 \pP (\upath (\X ) \in \cdot | \X = \x ) = \PPx 
\quad \text{ for $ \x = \lab (\xx ) $}
.\end{align*}
Here $ \upath $ is the unlabel path map given by $ \upath (\w )_t = \sum_i \delta_{w_t^i}$, $ \w = (w^i)$, and $ \{ \PPx \}$ is the diffusion 
associated with the Dirichlet form $ (\E , \dom )$ on $ \Lm $ 
as in \eqref{:11n}. We set 
\begin{align*}&
 \PPm = \int _{\sSS } \PPx d\mu 
.\end{align*}
Then, $ \PPm $ is a probability measure on $ C([0,\infty);\sSS ) $. 
The probability measure $ \PPm $ defines the $ \mu $-reversible diffusion on $ \sSS $. 
From Eq.\,\eqref{:11n}, note that the state space of the diffusion $( \XX , \PPm ) $ is restricted on $ \Ssi $. 
%

Let $ \Tt $ be the semi-group on $ \Lm $ associated with the Dirichlet form $ (\E , \dom )$. 
Clearly, $ \Tt 1 = 1 $ for all $ t $ because $ \Tt $ is $ \mu $-reversible. 
We present the inverse of this fact. 
We show that the Dirichlet form $ (\E , \dom )$ and the semi-group $ \Tt $ are ergodic in the following sense. 
\begin{theorem}	\label{l:12} 
\thetag{1} If $ f \in \dom $ and $ \E (f , f ) = 0 $, then $ f $ is constant $ \mu $-a.s. 
\\\thetag{2} 
If $ f \in \Lm $ is such that $ \Tt f = f $ $ \mu $-a.s.\,for all $ t $, 
then $ f $ is constant $ \mu $-a.s. 
\end{theorem}

Because of reversibility, we extend $ \PPm $ to the probability measure on 
$ \WW $ and denote it by the same symbol $ \PPm $. 
Let $ \map{\theta_t}{\WW }{\WW }$ be the shift such that 
$ \theta_t (\ww ) = \ww (\cdot + t )$. 
Then, $ \PPm $ is the invariant probability measure of $ \theta_t $. 
\begin{theorem}	\label{l:14}
$ \PPm $ is ergodic under the shift $ \theta_t $. 
That is, for any $ A \in \mathscr{B}( \WW ) $ such that 
$ \theta_t (A) = A $ for all $ t \in \R $, it holds that $ \PPm ( A ) \in \{ 0 , 1 \} $. 
\end{theorem}

We now explain the idea behind the proofs of the main theorems. 
The critical step is \tref{l:12} \thetag{1}; the other steps follow from this using rather standard argument. 
We use the lower Dirichlet form $ (\El , \doml )$ introduced in \lref{l:43}. 
This Dirichlet form satisfies the relation \cite{o.dfa,k-o-t.udf}
\begin{align}\label{:14a}&
 (\El , \doml ) \le (\E , \dom )
\end{align}
and has a finite volume approximation $ \{ (\ERl , \dRl ) \}_{\rR \in \N } $ such that 
$ (\El , \doml ) $ is the increasing limit of $ \{ (\ERl , \dRl ) \}_{\rR \in \N } $ (see \lref{l:43}). 
Each $ (\ERl , \dRl ) $ is given by the integration of Dirichlet forms 
$ (\1  , \doml _{\Rxi }^{\m } ) $ [see Eq.\,\eqref{:42w}]. 

In Eq.\,\eqref{:42a}, we relate $ (\1  , \doml _{\Rxi }^{\m } ) $ to $ (\ERxim , \domlRxim ) $.
Using the quasi-Gibbs property of $ \mu $ in \lref{l:31}, 
we prove the ergodicity of $ (\ERxim , \domlRxim ) $. 
The ergodicity of $ (\ERxim , \domlRxim ) $ implies that of 
$ (\1  , \doml _{\Rxi }^{\m } ) $. 
Then, using the number rigidity of $ \mu $ in \lref{l:42} and the tail triviality of $ \mu $ in \lref{l:44}, we deduce the ergodicity of the increasing limit $ (\El , \doml ) $. 

Because of the uniqueness of Dirichlet forms given by \lref{l:46}, 
the equality holds in Eq.\,\eqref{:14a}. That is, $ (\El , \doml ) = (\E , \dom )$. 
Hence, we obtain the ergodicity of $ (\E , \dom )$ from that of $ (\El , \doml ) $.

\medskip

Let $ \map{\Phi }{\Rd }{\R \cup \{ \infty \} }$ and 
$ \map{\Psi }{\Rd \times \Rd }{\R \cup \{ \infty \} }$ 
be measurable functions. A stochastic process given by a solution $ \X = (X^i)_{i}$ of the ISDE 
\begin{align*}&
X_t^i - X_0^i = B_t^i + \half \int_0^t \nabla \Phi (X_u^i) du 
+ \half \int_0^t \sum_{j\ne i} \nabla \Psi (X_u^i,X_u^j) du
\end{align*}
is called an interacting Brownian motion (in infinite dimensions) with potential 
$ (\Phi , \Psi )$. Here, $ (\nabla \Psi )(x,y)= \nabla_x \Psi (x,y)$. 
The study of interacting Brownian motions was initiated by Lang \cite{lang.1,lang.2}, who 
solved the above ISDE for $ (0, \Psi )$, where $ \Psi \in C_0^3(\Rd)$ is of Ruelle's class in the sense that it is super-stable and regular. 
Fritz \cite{Fr} constructed non-equilibrium solutions for the same potentials as in \cite{lang.1,lang.2} under the further restriction that the dimension $ d \le 4 $. 
Tanemura derived the solution for a hard-core potential \cite{T2}, while 
Fradon--Roelly--Tanemura solved the ISDE for the hard-core potential with long-range interactions, but still of Ruelle's class \cite{frt.00}. Various ISDEs with logarithmic interaction potentials have also been solved \cite{h-o.bes,k-o-t.udf,o.isde,o.rm2,o-t.tail,o-t.airy,tsai.14,kawa.22}. 

There are fewer results for the irreducibility and ergodicity of solutions of interacting Brownian motions. Albeverio--Kondratiev--R\"{o}ckner \cite{akr.98} proved the equivalence of the ergodicity of Dirichlet forms and the extremal property of the associated (grand canonical or canonical) Gibbs measures with potentials of Ruelle's class \cite{ruelle.2}. 
Corwin and Sun \cite{corwin-sun.14} proved the ergodicity of the Airy line ensembles, for which the dynamics are related to the Airy$ _2$ random point field. 
The first author and Tsuboi \cite{o-tu.irr} proved that the labeled dynamics of Dyson's model in infinite dimensions are irreducible. 
A general result concerning the ergodicity of Dirichlet forms can be found in \cite{fot.2}. 

The remainder of this paper is organized as follows. 

In \sref{s:2}, we recall the concept of the logarithmic derivative of sine$ _2$ random point field. 
In \sref{s:3}, we show that a labeled diffusion in a finite volume is ergodic. 
In \sref{s:4}, we prove \tref{l:12}. 
Finally, in \sref{s:5}, we prove \tref{l:14}. 

\section{ISDE and logarithmic derivative}\label{s:2}
Let $ \muone (dxd\sss ) = \rho ^1 \mu_x (d\sss )$ be the reduced one-Campbell measure of $ \mu $. Here, $ \rho^1 (x) \equiv \rho $ is the one-point correlation function of $ \mu $ and 
$ \mu_x = \mu (\cdot | \sss (\{ x \}) \ge 1 ) $ is the reduced Palm measure conditioned at $ x \in \R $. Here $ \rho $ is the constant in \eqref{:10f}.  

Let $ \dmu $ be the logarithmic derivative of $ \mu $. 
By definition, $ \dmu $ is the function defined on 
$ \mathbb{R}\times \sSS $ such that $ \dmu \in \Lmlocone $ and 
\begin{align} & \label{:20a}
 \int_{\mathbb{R} \times \sSS } \dmu ( s , \sss ) \varphi ( s , \sss ) d \muone 
=
- \int_{\mathbb{R} \times \sSS } \nabla \varphi ( s , \sss ) d \muone 
\end{align}
for all $ \varphi \in C_0^{\infty}(\mathbb{R}) \otimes \di ^b $, 
where $ \di ^b $ is the set consisting of bounded, local, and smooth functions on $ \sSS $ \cite{o-t.tail}. We write $ \sss = \sum_i \delta_{\si }$. 
It has been proved \cite{o.isde} that $ \mu $ has a logarithmic derivative such that, 
strongly in $ \Lmloctwo $, 
\begin{align}\label{:20b}
\dmu (s , \sss ) & = 2 \limR \sum_{{\si }\in \SR } \frac{1}{s -{\si }} 
= 2 \limR \sum_{ | s -{\si }| < \rR } \frac{1}{s -{\si }} 
.\end{align}
The convergence of sums in Eq.\,\eqref{:20b} follows from the fact that 
$ \mu $ is translation-invariant, $ d=1$, and the variance of $ \sss ([- \rR , \rR ])$ under $ \mu $ increases logarithmically as $ \rR \to \infty $. 
The translation invariance of $ \mu $ is clear because the determinantal kernel $ \KK $ in Eq.\,\eqref{:10f} defining $ \mu $ is translation-invariant. The logarithmic growth of the variance follows from a direct calculation using the Fourier transform \cite{Sos00}. 
The second equality in Eq.\,\eqref{:20b} comes from $ d = 1 $ and the translation invariance of $ \mu $. 

Using $ \dmu $, we represent the ISDE \eqref{:10e} for $ \beta = 2 $ as 
\begin{align} &\notag 
 X_t^i - X_0^i = B_t^i + 
\int_0^t \dmu (X_u^i , \XX _u^{i\diamondsuit }) du \quad (i\in\mathbb{Z})
,\end{align}
where $ \XX _u^{i\diamondsuit }= \sum_{j\ne i ,\, j \in \Z } \delta_{X_u^j}$. 
Recall that $ \pPi = \pP \circ \X ^{-1}$ and $ \X _0 \elaw \mu \circ \lab ^{-1}$ with the label $ \lab $ given in \sref{s:1}. 
The following irreducibility of labeled dynamics was proved in \cite{o-tu.irr}. 
\begin{lemma}[\cite{o-tu.irr}]\label{l:21}
 $ \pPi $ is irreducible. That is, if 
$ \mathbf{A}$ and $ \mathbf{B} \in \mathscr{B}(\RZZ ) $ satisfy 
\begin{align}\label{:21a}& 
 \pPi ( \wz \in \mathbf{A} ,\, \w _t \in \mathbf{B} ) = 0 
,\end{align}
then $\pPi ( \wz \in \mathbf{A} ) = 0 $ or $\pPi ( \w _t \in \mathbf{B} ) = 0 $. 
\end{lemma}
%
%
%

\section{Ergodicity of the local labeled diffusion}\label{s:3}

In this section, we consider a labeled $ m $-particle diffusion in $ [-\rR ,\rR ] $. 

Let $ \SR = \{ s \in \R \,;\, |s| < \rR \} $ and $ \piRc (\sss ) = \sss (\cdot \cap \SR ^c )$. 
Let $ \muRxim $ be the regular conditional probability defined by 
\begin{align}\label{:A1a}&
 \muRxim = \mu ( \,\cdot\, | \sss (\SR ) = m , \piRc (\sss ) = \piRc (\xi ) )
.\end{align}

We set the subset $ \SSSR $ of $ [-\rR , \rR ]^m $ as $ \SSSRone = [-\rR ,\rR ]$ and, 
for $ m \ge 2 $, 
\begin{align}\label{:A10}&
 \SSSR = \{ \xm = (x^i)_{i=1}^m \, ;\, - \rR \le x^i < x^{i+1} \le \rR \, , \, i=1,\ldots, m-1 \} 
.\end{align}
Let $ \SSRm = \{ \sss \in \sSS \, ;\, \sss (\SR ) = m \} $. 
Let $\map{ \labR }{ \SSRm}{\SSSR }$ be the map such that $ \labR (\xx ) = (x^1,\ldots,x^m)$, 
where $ \xx = \sum_{i=1}^m \delta_{x^i}$. 
A function $ \mmm (\xm ) $ is called the local density function of $ \mu $ on $ \SSSR $ 
conditioned as $ \sss (\overline{\sS }_{\rR } ) = m , \piRc (\sss ) = \piRc (\xi ) $ if 
$ \mmm d\xm $ is the the image measure of $ \muRxim $ under the map $ \labR $. 
That is, 
\begin{align}\label{:A10b}&
 \mmm d\xm = \muRxim \circ (\labR )^{-1}
.\end{align}
 
Let $ \Deltam (\x )$ be the difference product such that, for $ \x = (x^1,\ldots,x^m )$, 
\begin{align}\notag 
\Deltam (\x ) =& \prod_{i<j}^{m}|x^i-x^j|^2 && \text{ for } m \ge 2 
\\ \notag 
= & 1 && \text{ for } m= 1
.\end{align}
In \cite{o.rm}, the first author proved that $ \muRxim $ has 
a local labeled density $ \mmm $ satisfying the estimate of Eq.\,\eqref{:A1b}. 
\begin{lemma} [{Theorem 2.2 in \cite{o.rm}}] \label{l:31}
For each $ \rR \in \N $ and $ \mu $-a.s.\,$ \xi $, 
the probability measure $ \muRxim \circ (\labR )^{-1}$ has a density 
$ \mmm $ with respect to the Lebesgue measure such that, 
for $ \xm = (x^i)_{i=1}^{\m } \in \SSSR $, 
\begin{align}\label{:A1b}&
\cref{;A1}^{-1}\Deltam (\x ) \le 
 \mmm (\xm ) \le 
\cref{;A1}\Deltam (\x )
.\end{align}
Here, $ \Ct = \cref{;A1}({\Rpixi }) \label{;A1}\ge 1 $ is a constant depending on 
$ (\Rpixi )$ for $ \mu $-a.s.\,$ \xi $. 
\end{lemma}
Taking Eq.\,\eqref{:A1b} into account, we set 
\begin{align} \label{:32p}& 
 \EWRD (f,g) = \int_{\SSSR } \DDDRO [f,g] (\x ) \Deltam (\x ) d\x 
,\\ \label{:32q} &
\DDDRO [f,g] = \half \sum_{i=1}^m \PD{f}{x^i} \PD{g}{x^i}
.\end{align}

\begin{lemma} \label{l:32}
\thetag{1} $ (\EWRD , C_0^{\infty}(\SSSR ))$ is closable on $ \LRm $. 
\\\thetag{2} 
The diffusion associated with the closure $ ( \EWRD , \domWR )$ 
of $ (\EWRD , C_0^{\infty}(\SSSR ))$ on $ \LRm $ 
has a transition probability density of 
$ \pWR $, which is smooth on $ \SSSR $ and has the property that 
\begin{align}\label{:32a} &
0 < \pWR (t,\x , \y ) < \infty 
\quad \text{ for } (t,\x , \y ) \in (0,\infty )\times \SSSR \times \SSSR 
.\end{align}
\end{lemma}
\begin{proof}
Because $ \Deltam (\x )$ is continuous, we obtain \thetag{1} (see \cite[Lemma 3.2]{o.dfa}).

Let $ \Gamma = \{ \x =(x^i)_{i=1}^m; x^i=x^j \text{ for some $ i\ne j$} \} $. 
Then we can easily see that $ \mathrm{Cap}(\Gamma ) = 0 $, 
where $ \mathrm{Cap}$ is the capacity of 
the Dirichlet form $ ( \EWRD , \domWR )$ on $ \LRm $. 

The transition density $ \pWR $ is described by the heat equation 
\begin{align}\label{:32b}&
\Big\{
\PD{}{t} - \half \sum_{i=1}^m ( \PD{}{x^i})^2 - \sum_{i,j=1 , i\ne j }^{m} 
\frac{1}{x^i-x^j} \PD{}{x^i} \Big\} \pWR (t,\x ,\y ) = 0 
.\end{align}
Here Neumann boundary condition is posed on $ \{ x^1= - R \} $ and $ \{ x^m= R \} $. 
The boundary $ \Gamma $ can be ignored because $ \mathrm{Cap}(\Gamma ) = 0 $. 

Note that $ \SSSR $ is a relatively compact, connected open set in $ [-\rR , \rR ]^m $, and that 
the coefficients are smooth on $ \SSSR $. 
Hence, solving \eqref{:32b}, we obtain \thetag{2}. 
\end{proof}

\begin{lemma} \label{l:33}
\thetag{1} Suppose $ \TW (t)f = f $ for all $ t \ge 0 $. Then, $ f $ is constant a.e. 
\\\thetag{2} 
Suppose $ f \in \domWR $ satisfies $ \EWRD (f,f) = 0 $. 
 Then, $ f $ is constant a.e. 
\end{lemma}
\begin{proof}
Note that $ \SSSR $ is a relatively compact, connected open set in $ [-\rR , \rR ]^m $. 
Then, \lref{l:33} follows from Eq.\,\eqref{:32a}. 
\end{proof}

We set 
\begin{align} &\label{:34z}
\ERxim (f,g) = \int_{\SSSR } \DDDRO [f,g] (\x ) \mmm (\x ) d\x 
.\end{align}

\begin{lemma} \label{l:34}
 $ (\ERxim , C_0^{\infty}(\SSSR ))$ is closable on $ \Lmmm $. 
\end{lemma}

\begin{proof}
Using Eq.\,\eqref{:A1b}, we deduce 
\begin{align} &\notag 
\cref{;A1}^{-1} \EWRD (f,f)\le 
 \ERxim (f,f) \le
\cref{;A1} \EWRD (f,f)
,\\ &\notag 
 \cref{;A1}^{-1} (f,f)_{\LRm }
 \le 
(f,f)_{\Lmmm }
 \le 
\cref{;A1} (f,f)_{\LRm } 
.\end{align}
Hence, \lref{l:34} follows from \lref{l:32} \thetag{1}. 
\end{proof}

Let $ (\ERxim , \domlRxim ) $ be the closure of 
$ (\ERxim , C_0^{\infty}(\SSSR ))$ on $ \Lmmm $. 
Let $ \TRxim (t)$ be the associated $ L^2$-semi-group on 
$ \Lmmm $.

\begin{theorem}	\label{l:35}
\thetag{1} 
Suppose that $ f \in \domlRxim $ satisfies $ \ERxim (f,f) = 0 $. 
Then, $ f $ is constant $ \mu $-a.s. 
\\\thetag{2} 
Suppose that $ f \in \Lm $ satisfies $ \TRxim (t)f = f $ for all $ t $. 
Then, $ f $ is constant $ \mu $-a.s. 
\end{theorem}
\begin{proof}
Statement \thetag{1} follows from \lref{l:34} and \lref{l:33} \thetag{2}. 
Suppose that $ \TRxim (t)f = f $ for all $ t $. Then, using the relation 
\begin{align} &\notag 
(f,f)_{\Lmmm } - ( \TRxim (t)f, \TRxim (t)f)_{\Lmmm } 
\\ \notag & = 
\int_0^t \ERxim ( \TRxim (u)f, \TRxim (u)f) du 
,\end{align}
we obtain 
\begin{align*}&
\text{$ \ERxim ( \TRxim (u)f , \TRxim (u)f) = 0 $ for a.e.\,$ u \ge 0$}
.\end{align*}
Hence, $ \TRxim (u)f $ is a constant function for a.e.\,$ u $. 
Taking $ u \to 0 $, we see that $ \TRxim (u)f $ converges strongly to $ f $ in $ \Lmmm $. 
Hence, we deduce that $ f $ is a constant function. 
This proves \thetag{2}. 
\end{proof}

\section{Proof of \tref{l:12}. } \label{s:4}

Let $ \muRxim $ be as in \eqref{:A1a}. 
It is clear that $ \mu $ has a disintegration such that 
\begin{align}\label{:42x}&
\muR ^{-1} (\cdot )= \sum_{m=0}^{\infty}\mu (\SSRm ) \int_{\sSS } \muRxim (\cdot ) \mu (d \xi )
.\end{align}
Let $ \DDDR $ be the carr\'{e} du champ on $ \sSS $ such that 
\begin{align}\label{:42y} &
\DDDR [f,g] (\sss )= \half \sum_{\si \in \SR }
\PD{\f }{\si }(\mathbf{s}) \PD{\g }{\si }(\mathbf{s})
.\end{align}
Here, for $ f \in \di $, the function $ \f $ is symmetric in $ \mathbf{s}$ 
such that $ \f (\mathbf{s}) = f (\sss )$ 
for $ \mathbf{s} = (\si )_i $ and $ \sss = \sum_i \delta_{\si }$. 
Note that the right-hand side can be regarded as function in $ \sss $ because it is a symmetric function in $ \mathbf{s}$. 
We set 
\begin{align}\label{:42z}&
 \1 (f,g) = \int_{\sSS } \DDDR [f,g] (\sss )\muRxim (d\sss )
,\\ \notag 
& \dom \Rxim = \{ f \in \di \, ;\, \1 (f,f) < \infty,\, f \in L^2 (\sSS , \muRxim ) \} 
.\end{align}

From \eqref{:A10b}, \eqref{:32q}, \eqref{:34z}, \eqref{:42y}, and \eqref{:42z}, we see that 
$ \ERxim $ and $ \1 $ are isometric in the sense that for $ f \in \di $
\begin{align}\label{:42a}&
\ERxim (\fRxi , \fRxi ) = \1  (f,f) \quad \text{ for $ \mu $-a.s.\,$ \xi $}
.\end{align}
Here, for a function $ f $ on $ \sSS $, we set $ \fRxi (\x ) = f (\sss ) $ and 
$ \ulab (\x ) + \piRc (\xi ) = \sss $, where 
$ \ulab (\x ) = \sum_{i=1}^m \delta_{x^i} $ for $ \x = (x^i)_{i=1}^m$ as before. 

Using \lref{l:34} and \eqref{:42a}, we deduce that 
$ (\1 , \dom \Rxim )$ is closable on $ L^2 (\sSS , \muRxim )$. 
We denote the closure of $ ( \1 , \dom \Rxim )$ on $ L^2 (\sSS , \muRxim )$ by 
$ ( \1 , \dlRxim ) $. 

Let 
\begin{align}\notag 
& \3 (f,g) = \int_{\sSS } \DDDR [f,g] d\mu 
,\\ \notag 
& \diR = \{ f \in \di \, ;\, \ER (f,f) < \infty,\, f \in \Lm \} 
.\end{align}
We quote a result from \cite{o.rm}. 
\begin{lemma}[\cite{o.rm}] \label{l:41}
$ (\3 , \diR )$ is closable on $ \Lm $. 
\end{lemma}
\begin{proof}
\lref{l:41} follows from Lemmas 3.3--3.5 and Theorem 2.2 in \cite{o.rm}. 
\end{proof}

Let $ (\ERl , \dRl ) $ be the closure of $ (\3 , \diR )$ on $ \Lm $. 
Using Eqs.\,\eqref{:42x}, \eqref{:42y}, and \eqref{:42z}, we deduce 
\begin{align}\label{:42w}&
\ERl = \sum_{m=0}^{\infty} \mu (\SSRm )
\int_{\sSS } \1 \mu (d\xi ) 
,\quad 
\dRl \subset \bigcup_{m=0}^{\infty} \Big\{ \bigcap_{\mu \text{-a.s.}\, \xi } \dlRxim \Big\} 
.\end{align}

We recall the concept of number rigidity defined in \cite{ghosh.2015}. 
We say that a random point field $ \nu $ is number rigid if, for any relatively compact set $ A $, it holds that $ \sss ( A ) $ is non-random under $ \nu (\cdot | \pi_{A^c}(\xi ) )$ 
for $ \nu \circ \pi_{A^c}$-a.s.\,$ \sss $, where 
$ \pi_{A^c} (\sss ) = \sss (\cdot \cap A^c  )$. 

This means that the number of particles inside $ A $ is uniquely determined by information outside $ A $ for $ \nu $-a.s.\,$ \sss $. 
(See \cite{buf.2016} for a historical remark on rigidity.) 
We quote the following result from \cite{ghosh.2015}. 
\begin{lemma}[\cite{ghosh.2015}] \label{l:42}
The sine$ _2$ random point $ \mu $ is number rigid. 
\end{lemma}

Because of this property of number rigidity, 
we can find a unique $ \mxi \in \N $ for each $ \rR $ and $ \mu $-a.s.\,$ \xi $ such that 
$\mxi =  \sss (\SR ) $ for $ \mu _{\Rxi } $-a.s.\,$ \sss $, 
where $ \mu _{\Rxi }= \mu (\piR (\sss ) \in \cdot | \piRc (\sss ) = \piRc (\xi ))$ is the regular conditional probability. Hence, we have 
\begin{align}\label{:43v}&
\mu = \int_{\sSS } \muRximxi \mu (d \xi )
.\end{align}
Furthermore, $ \mxi $ is $ \sigma [\piRc ] $-measurable for each $ \rR \in \N $. 
Using Eq.\,\eqref{:43v}, we reduce Eq.\,\eqref{:42w} to the following. 
\begin{align}\label{:43w}&
\ERl = \int_{\sSS } \El _{\Rxi }^{\mxi } \mu (d\xi )
,\quad 
\dRl \subset \bigcap _{\text{$ \mu $-a.s.\,$ \xi $}} \domlRxi ^{\mxi } 
.\end{align}
Note that $ \doml \subset \dRl \subset \domlRxi ^{\mxi } $ and that 
$ \El _{\Rxi }^{\mxi } (f,f)$ is $ \sigma [\piRc ] $-measurable for each $ f \in \doml $.

For the Dirichlet forms $ (\E ^1, \dom ^1 )$ and $ (\E ^2, \dom ^2 )$, we write 
$ (\E ^1, \dom ^1 ) \le (\E ^2, \dom ^2 ) $ if 
$ \E ^1 (f,f) \le \E ^2 (f,f) $ for all $f \in \dom ^2$ and $ \dom ^1 \supset \dom ^2 $. 
We say that $ \{ (\E ^i, \dom ^i) \}_{i} $ is increasing if 
$ (\E ^i, \dom ^i) \le (\E ^{i+1}, \dom ^{i+1})$ for all $ i $. 
We now quote a result from \cite{o.dfa}. 
\begin{lemma}[{\cite[Lemma 2.2]{o.dfa}}] \label{l:43}
\thetag{1} $ \{ (\ERl , \dRl ) \}_{\rR \in \N } $ is increasing. 
\\\thetag{2} Let 
$ (\El , \doml ) $ be the increasing limit of $ \{ (\ERl , \dRl ) _{\rR \in \N } $, 
where 
\begin{align*}&
\doml = \{ f \in \bigcap_{\rR = 1}^{\infty} \dRl 
\, ;\, \limR \ERl (f,f) < \infty \} 
.\end{align*}
Then, $ (\El , \doml ) $ is the closed form on $ \Lm $. 
\end{lemma}

We set $ \TTR = \sigma [\piRc ]$. 
Let $ \TT $ be the tail $ \sigma $-field of $ \sSS $, defined as 
\begin{align} & \notag 
\TT = \bigcap_{\rR = 1}^{\infty} \TTR 
.\end{align}
We quote the following result from \cite{o-o.tail, ly.18,bqs}. 
\begin{lemma} \label{l:44} 
The sine$ _2$ random point field $ \mu $ is tail-trivial. 
That is, $ \mu (\mathfrak{A}) \in \{ 0,1 \} $ for any $ \mathfrak{A}\in \TT $. 
\end{lemma}

Let $ \Tlt $ be the $ L^2$-semi-group associated with the Dirichlet form $ (\El , \doml ) $ 
on $ \Lm $. 

\begin{theorem}	\label{l:45} 
\thetag{1} If $ f \in \doml $ and $ \El (f , f ) = 0 $, then $ f $ is constant $ \mu $-a.s. 
\\\thetag{2} 
If $ f \in \Lm $ is such that $ \Tlt f = f $ $ \mu $-a.s.\,for all $ t $, 
then $ f $ is constant $ \mu $-a.s. 
\end{theorem}
\begin{proof} 
Let $ f \in \doml $ such that $ \El (f , f ) = 0 $. Then, using 
\lref{l:43}, we deduce that $ f \in \dRl $ and $ \ERl (f,f) = 0 $ for all $ \rR \in \N $. 
Using this and Eq.\,\eqref{:43w}, we see that 
\begin{align}\label{:45a}&
 \El \Rximxi (f,f) = 0 \quad \text{ for $ \mu $-a.s.\,$ \xi $}
.\end{align}
Hence, from Eqs.\,\eqref{:42a} and \eqref{:45a}, we deduce that 
\begin{align}\label{:45b}&
\ERxi ^{\mxi } (\fRxi , \fRxi ) = 0 \quad \text{ for $ \mu $-a.s.\,$ \xi $}
.\end{align}
Here, for a function $ f $ on $ \sSS $, we set $ \fRxi (\x ) = f (\sss ) $ and 
$ \ulab (\x ) + \piRc (\xi ) = \sss $. 

Using \tref{l:35} \thetag{1} and Eq.\,\eqref{:45b}, we deduce that 
$ \fRxi (\x ) = f (\sss ) $ is constant in $ \x $ for $ \mu $-a.s.\,$ \sss $. 
Thus, we see that $ f $ is a $\TTR $-measurable function. 

Taking $ \rR \to \infty $, we deduce that $ f $ is $ \TT $-measurable. 
Hence, we deduce that $ f $ is constant for $ \mu $-a.s.\,$\sss $ because $ \TT $ is $ \mu $-trivial by \lref{l:44}. 
Thus, we have proved \thetag{1}. 

For all $ t \ge 0 $, note that 
\begin{align} &\notag 
(f , f )_{\Lm } - (\Tlt f , \Tlt f ) = \int_0^t \El (\Tlu f , \Tlu f ) du 
.\end{align}
Using this and the assumption, we obtain $ (f , f )_{\Lm } - (\Tlt f , \Tlt f ) = 0 $ for all $ t $. 
Hence, $ \El (\Tlu f , \Tlu f ) = 0 $ for a.e.\,$ u \in (0,\infty )$ with respect to the Lebesgue measure. 
Hence, using \thetag{1}, we deduce that 
$ \Tlu f (\sss ) $ is a constant function in $ \sss $ for a.e.\,$ u \in (0,\infty )$. 
Recall that $ \Tlu f $ converges to $ f $ in $ \Lm $ as $ u \to 0 $. 
Using these facts, we have proved \thetag{2}. 
\end{proof}

\begin{lemma}[{\cite[Section 7.1]{k-o-t.udf}}] \label{l:46}
$ (\El , \doml ) = (\E , \dom )$. 
\end{lemma}

\noindent {\em Proof of \tref{l:12}. }
We immediately obtain \tref{l:12} from \tref{l:45} and \lref{l:46}. 
\qed

\section{Proof of \tref{l:14}. }\label{s:5}

Let $ \mathscr{I} $ be the set consisting of the invariant probability measures of $ \Tt $. 
\begin{lemma}	\label{l:51}
$ \mu $ is extremal in $ \mathscr{I} $. 
That is, if there exist $ \mu_1 , \mu_2 \in \mathscr{I} $ and constants 
$\alpha_1$ and $ \alpha_2 $ such that 
\begin{align}\label{:51a}&
\mu = \alpha_1 \mu_1 + \alpha_2 \mu_2 , \quad 
0 < \alpha_1 , \alpha_2 < 1 
,\end{align}
then $ \mu = \mu_1= \mu_2 $. 
\end{lemma}
\begin{proof}

From Eq.\,\eqref{:51a}, we have that both $ \mu_1 $ and $ \mu_2 $ are absolutely continuous with respect to $ \mu $. Hence, there exist non-negative functions $ m_i $ such that 
$ d \mu_i = m_i d\mu $, $ i=1,2$. 
Using this and Eq.\,\eqref{:51a}, we deduce 
\begin{align}&\notag 
d\mu = \alpha_1 m_1d\mu + \alpha_2 m_2d\mu 
.\end{align}
This implies that $ 1 = \alpha_1 m_1 + \alpha_2 m_2$. 
Hence, $ 0 \le m_i \le 1/\alpha_i$ for $ i = 1,2 $. 

Using $ \mu_i \in \mathscr{I} $, we have $ \Tt m_i = m_i $ for all $ t $. 
Hence, we deduce $ m_i = 1 $ from \tref{l:12} and 
$ \int_{\sSS } m_i d\mu = 1 $ for $ i=1,2$. 
This implies that $ \mu = \mu_1 = \mu_2 $.
\end{proof}

\noindent {\em Proof of \tref{l:14}. } 
It is well known that $ \mu $ is extremal in $ \mathscr{I} $ if and only if 
$ \PPm $ is ergodic under the shift $ \theta_t $ (see Theorem 3.8 in \cite{rey}). 
\tref{l:14} follows from this and \lref{l:51}. \qed

\section{Data availability statements}
The data that supports the findings of this study are available within the article. 

\section{ Acknowledgment}
{This work was supported by JSPS KAKENHI (Grant Nos JP16H06338, JP20K20885, JP21H04432, and JP18H03672). 
We thank Stuart Jenkinson, PhD, from Edanz (https://jp.edanz.com/ac) for editing a draft of this manuscript.}

{
\small 
\noindent 
Hirofumi Osada\\
Faculty of Mathematics, Kyushu University, \\ Fukuoka 819-0395, Japan. \\
\texttt{osada@math.kyushu-u.ac.jp} 

\bs

\noindent 
Shota Osada 
\\
Institute of Mathematics for Industry, 
Kyushu University, \\ Fukuoka 819-0395, Japan. \\
\texttt{s-osada{@}imi.kyushu-u.ac.jp}


\begin{thebibliography}{99}

 \DN\ttl[1]{:\,\textit{#1}. } 



\bibitem{akr.98} Albeverio, S., Kondratiev, Yu. G., R\"{o}ckner, M. \ttl{Analysis and geometry on configuration spaces: the Gibbsian case} J. Funct. Anal. 157 (1998), no. 1, 242--291. 










\bibitem{buf.2016} Bufetov, A.I.\ttl{Rigidity of determinantal point processes with the Airy, the Bessel and the Gamma kernel} Bull. Math. Sci. {\bf 6}, 163--172 (2016) doi: https://doi.org/10.1007/s13373-015-0080-z

\bibitem{bqs} Bufetov I.A., Qiu Yanqi, Shamov A.\ttl{Kernels of conditional determinantal measures and the proof of the Lyons-Peres conjecture} J. Eur. Math. Soc. 23 (2021), 1477-1519. doi: 10.4171/JEMS/1038.

\bibitem{corwin-sun.14} Corwin, I., Sun, Xin.\ttl{Ergodicity of the Airy line ensemble} Electric. Commun. Probab. {\bf 19} no.49, 1--11 (2014). 

\bibitem{Dys62} Dyson, F. J.\ttl{A Brownian-motion model for the eigenvalues of a random matrix} J. Math. Phys. {\bf 3}, 1191--1198 (1962).




\bibitem{frt.00} Fradon, M., Roelly, S., Tanemura, H.\ttl{An infinite system of Brownian balls with infinite range interaction} Stochastic Process. Appl. 90, no. 1, 43--66 (2000). 


\bibitem{Fr} Fritz,~J.\ttl{Gradient dynamics of infinite point systems} Ann. Probab. {\bf 15 } 478--514 (1987). 

\bibitem{fot.2} Fukushima, M., Oshima, Y., Takeda M.\ttl{Dirichlet forms and symmetric Markov processes} 2nd ed., Walter de Gruyter (2011). 


\bibitem{ghosh.2015} Ghosh, S.\ttl{Determinantal processes and completeness of random exponentials: the critical case} Probab. Theory Related Fields {\bf 163} (2015), no. 3-4, 643--665.

\bibitem{h-o.bes} Honda,~R., Osada,~H.\ttl{Infinite-dimensional stochastic differential equations related to Bessel random point fields} Stochastic Processes and their Applications 125, no. 10, 3801--3822 (2015). 

\bibitem{IW} Ikeda, N., Watanabe, S.\ttl{Stochastic differential equations and diffusion processes} 2nd ed, North-Holland (1989). 










\bibitem{kawa.22} Kawamoto, Y.\ttl{Interacting Brownian motions in infinite dimensions related to the origin of the spectrum of random matrices} Modern Stoch. Theory Appl.(2022), 1-34, DOI 10.15559/21-VMSTA193. 

\bibitem{k-o.fpa} Kawamoto, Y., Osada, H.\ttl{Finite particle approximations of interacting Brownian particles with logarithmic potentials} J. Math. Soc. Japan, Volume 70, Number 3 (2018), 921--952. doi:10.2969/jmsj/75717571. 


\bibitem{k-o.sg} Kawamoto, Y., Osada, H. \ttl{Dynamical Bulk Scaling Limit of Gaussian Unitary Ensembles and Stochastic Differential Equation Gaps} J Theor Probab 32, 907--933 (2019). https://doi.org/10.1007/s10959-018-0816-2

\bibitem{k-o.du} Kawamoto, Y., Osada, H.\ttl{Dynamical universality for random matrices} (to appear in Partial Differ. Equ. Appl.)

\bibitem{k-o-t.udf} Kawamoto, Y., Osada, H., Tanemura, H.\ttl{Uniqueness of Dirichlet forms related to infinite systems of interacting Brownian motions} Potential Anal 55, 639--676 (2021). https://doi.org/10.1007/s11118-020-09872-2



\bibitem{k-o-t.ifc} Kawamoto, Y., Osada, H., Tanemura H., \ttl{Infinite-dimensional stochastic differential equations and tail $ \sigma $-fields II: the IFC condition} J. Math. Soc. Japan {\bf 74} (2022), 79--128. 


\bibitem{lang.1} Lang,~R.\ttl{Unendlich-dimensionale Wienerprocesse mit Wechselwirkung I} Z. Wahrschverw. Gebiete {\bf 38 } 55--72 (1977). 

\bibitem{lang.2} Lang,~R.\ttl{Unendlich-dimensionale Wienerprocesse mit Wechselwirkung II} Z. Wahrschverw. Gebiete {\bf 39 } 277--299 (1978). 



\bibitem{ly.18} Lyons,~R.\ttl{A Note on Tail Triviality for Determinantal Point Processes} Electron. Commun. Probab. {\bf 23}, (2018), paper no. 72, 1--3 ISSN: 1083--589X




\bibitem{Meh04}Mehta, M. L.\ttl{Random Matrices. 3rd edition}, Amsterdam: Elsevier (2004). 

%
%
%
%




\bibitem{o.dfa} Osada, H.\ttl{Dirichlet form approach to infinite-dimensional Wiener processes with singular interactions} Commun. Math. Phys. {\bf 176}, 117--131 (1996). 


\bibitem{o.col} Osada,~H.\ttl{Non-collision and collision properties of Dyson's model in infinite dimensions and other stochastic dynamics whose equilibrium states are determinantal random point fields} in Stochastic Analysis on Large Scale Interacting Systems, eds.\ T.\ Funaki and H.\ Osada, Advanced Studies in Pure Mathematics \textbf{39}, 325--343 (2004). 

\bibitem{o.tp} Osada,~H.\ttl{Tagged particle processes and their non-explosion criteria} J. Math. Soc. Japan, {\bf 62}, No.\ {\bf 3}, 867--894 (2010). 

\bibitem{o.isde} Osada,~H.\ttl{Infinite-dimensional stochastic differential equations related to random matrices} Probability Theory and Related Fields, {\bf 153}, 471--509 (2012). 

\bibitem{o.rm} Osada,~H.\ttl{Interacting Brownian motions in infinite dimensions with logarithmic interaction potentials} Ann. of Probab. {\bf 41}, 1--49 (2013). 

\bibitem{o.rm2} Osada, H.\ttl{Interacting Brownian motions in infinite dimensions with logarithmic interaction potentials II : Airy random point field} Stochastic Processes and their Applications {\bf 123}, 813--838 (2013). 



\bibitem{o-o.tail} Osada, H., Osada, S.\ttl{Discrete approximations of determinantal point processes on continuous spaces: tree representations and tail triviality} J Stat Phys (2018) 170: 421. https://doi.org/10.1007/s10955-017-1928-2



\bibitem{o-t.tail} Osada,~H., Tanemura,~H., \ttl{Infinite-dimensional stochastic differential equations and tail $ \sigma $-fields} Probab.\ Theory Relat.\ Fields {\bf 177} (2020), 1137--1242. https://doi.org/10.1007/s00440-020-00981-y. 

\bibitem{o-t.airy} Osada,~H., Tanemura,~H.\ttl{Infinite-dimensional stochastic differential equations arising from Airy random point fields} (preprint) {\sf arXiv:1408.0632 [math.PR] (ver.\,8)}.





\bibitem{o-tu.irr} Osada,~H., Tsuboi,~R.\ttl{Dyson's model in infinite dimensions is irreducible}(to appear in Festschrift in honor of Masatoshi Fukushima's Beiju). 





\bibitem{v-v} Valk\'{o}, B., Vir\'{a}g, B. Continuum limits of random matrices and the Brownian carousel. Invent. math. 177, 463--508 (2009). https://doi.org/10.1007/s00222-009-0180-z



\bibitem{rey} Rey-Bellet, L., \ttl{Ergodic properties of Markov processes} Open quantum systems. II, 1--39, Lecture Notes in Math., {\bf 1881}, Springer, Berlin, 2006. 



\bibitem{ruelle.2} Ruelle,~D.\ttl{Superstable interactions in classical statistical mechanics} Commun. Math. Phys. {\bf 18} 127--159 (1970). 




\bibitem{Sos00}Soshnikov, A.\ttl{Determinantal random point fields} Russian Math. Surveys {\bf 55}, 923--975 (2000). 

\bibitem{Spo87}Spohn, H.\ttl{Interacting Brownian particles:a study of Dyson's model} In: Hydrodynamic Behavior and Interacting Particle Systems, G. Papanicolaou (ed), IMA Volumes in Mathematics and its Applications, {\bf 9}, Berlin: Springer-Verlag, pp. 151--179 (1987). 

\bibitem{Spo.tracer} Spohn, H.\ttl{Tracer dynamics in Dyson's model of interacting Brownian particles} J Stat Phys (2017) {\bf 47}, 669--679 (1987). 

\bibitem{T2} Tanemura,~H.\ttl{A system of infinitely many mutually reflecting Brownian balls in $\mathbb{R} ^d $} Probab.\ Theory Relat.\ Fields {\bf 104} 399--426 (1996). 


\bibitem{tsai.14} Tsai, Li-Cheng.\ttl{Infinite dimensional stochastic differential equations for Dyson's model} Probab.\ Theory Relat.\ Fields {\bf 166}, 801--850 (2016).

 
\end{thebibliography}
\end{document}